\documentclass[11pt]{amsart}

\usepackage{amsmath,amssymb,amsthm,amscd}
\usepackage[frame,cmtip,arrow,matrix,line,graph,curve]{xy}
\usepackage{graphpap,color}
\usepackage{hyperref}

\def\dfrac{\displaystyle\frac}

\newtheorem{prop}{Proposition}[section]
\newtheorem{theo}[prop]{Theorem}
\newtheorem{lemm}[prop]{Lemma}
\newtheorem{coro}[prop]{Corollary}

\newtheorem*{rema}{Remark}

\renewcommand{\geq}{\geqslant}
\renewcommand{\le}{\leqslant}
\renewcommand{\ge}{\geqslant}

\newcommand{\tr}{\operatorname{tr}}

\allowdisplaybreaks[1]

\numberwithin{equation}{section}

\title{Interior $C^2$ Estimates for 2-Sum Hessian Equations}

\author{Changyu Ren}
\address{School of Mathematical Science\\
	Jilin University\\ Changchun\\ China}
\email{rency@jlu.edu.cn}
\author{Xianlu Sun}
\address{School of Mathematics and Statistics\\
	Northeast Normal University\\ Changchun\\ China}
\email{sunxl221@nenu.edu.cn}

\begin{document}

	\begin{abstract}
		This paper establishes an interior $C^2$ estimate for solutions of
		\[
		\sigma_2(D^2u)+\alpha(x)\sigma_1(D^2u)=f(x,u,Du)>0,
		\]
		on the positive branch $\Delta u>0$, under the assumptions $\alpha\ge0$ and $D^2\alpha=\beta I$. Using one extra variable and a quartic correction, the equation is reduced to a quadratic Hessian equation in one higher dimension. The terms involving $D^2u$ cancel, so the interior estimate of Li--Wu \cite{LW} applies. The gradient-independent case is also treated. Finally, an example shows that an $L^\infty$ bound on $\alpha$ alone is not enough.
	\end{abstract}
	\maketitle

	\section{Introduction}
	
	Consider the following sum Hessian equation in $B_2^n\subset\mathbb R^n$, $n\ge2$:
	\begin{equation}\label{eq:1.1}
		\sigma_2(D^2u)+\alpha(x)\sigma_1(D^2u)=f(x,u,Du),\qquad x\in B_2^n.
	\end{equation}
	Here $\sigma_k(D^2u)=\sigma_k(\lambda(D^2u))$ is the $k$-th elementary symmetric function of the eigenvalues of $D^2u$. Assume
	\[
	f(x,u,Du)>0,
	\]
	and
	\begin{equation}\label{eq:1.2}
		\alpha(x)=a_0+b\cdot x+\frac{\beta}{2}|x|^2\ge0,
		\qquad x\in B_2^n,
	\end{equation}
	where $a_0,\beta\in\mathbb R$ and $b\in\mathbb R^n$.
	
	The notation used below is as follows. Let $B_r^n(y)$ denote the ball in $\mathbb R^n$ centered at $y$ with radius $r$, and set $B_r^n:=B_r^n(0)$. Let $\operatorname{Sym}(n)$ denote the space of all $n\times n$ real symmetric matrices. For $A\in\operatorname{Sym}(n)$, denote its eigenvalues by
	\[
	\lambda(A)=(\lambda_1,\cdots,\lambda_n)
	\]
	and denote
	\[
	\sigma_k(A):=\sigma_k(\lambda(A)).
	\]
	In particular,
	\[
	\sigma_1(A)=\operatorname{tr}A=\sum_{i=1}^n\lambda_i,\qquad
	\sigma_2(A)=\sum_{1\le i<j\le n}\lambda_i\lambda_j.
	\]
	Accordingly, a function
	$u\in C^2(\Omega)$ is called $k$-admissible if
	\[
	\lambda(D^2u(x))\in\Gamma_k.
	\]
	The G\r{a}rding cone is
	\[
	\Gamma_k=\bigl\{\lambda\in\mathbb R^n\mid \sigma_m(\lambda)>0,\ m=1,\cdots,k\bigr\}.
	\]
	A symmetric matrix $A$ is said to belong to $\Gamma_k$ whenever $\lambda(A)\in\Gamma_k$. For $A=(a_{ij})$, let
	\[
	|A|^2=\sum_{i}\lambda_{i}^2(A).
	\]
	When $\alpha=0$, equation~\eqref{eq:1.1} becomes the $2$-Hessian equation
	\[
	\sigma_2(D^2u)=f(x,u,Du).
	\]

	The classical $k$-Hessian equation
	\begin{equation}\label{eq:1.3}
		\sigma_k(D^2u)=f(x,u,Du),\qquad x\in\Omega,
	\end{equation}
	and the related curvature equation
	\begin{equation}\label{eq:1.4}
		\sigma_k(\kappa( X))=\psi( X,\nu),\qquad  X\in M,
	\end{equation}
	are basic fully nonlinear equations in PDE and geometric analysis. They appear in many geometric problems; see, for example, \cite{CNS3,CNS4,GLM,GLL,GRW,O,P3} and the references therein.

	Interior $C^2$ estimates are a basic part of the regularity theory for fully nonlinear equations. For the Monge--Amp\`ere equation, Heinz \cite{H} proved the interior $C^2$ estimate in dimension $n=2$. In dimensions $n\ge3$, Pogorelov \cite{P3} showed that a general interior $C^2$ estimate may fail and obtained the well-known Pogorelov estimate.
	
	For $k$-Hessian equations, Urbas \cite{UE} extended Pogorelov's example and showed that, for $k\ge3$, solutions of $\sigma_k(D^2u)=f$ need not satisfy a general interior $C^2$ estimate. The case $k=2$ is therefore special, and interior $C^2$ estimates for the $2$-Hessian equation have received much attention.
	
	For $\sigma_2(D^2u)=1$, Warren--Yuan \cite{WY} proved the interior estimate in dimension $n=3$, and Shankar--Yuan \cite{RY2} treated dimension $n=4$. In higher dimensions, McGonagle--Song--Yuan \cite{MSY} obtained the estimate for almost convex solutions, while Shankar--Yuan \cite{RY1,RY2} considered semiconvex and dynamically semiconvex solutions. The dynamic semiconvexity condition is
	\begin{equation}\label{eq:1.5}
		\lambda_{\text{min}} \geq
		-c(n)\Delta u  \quad {\rm with}\quad  c(n) = \dfrac{\sqrt{3n^2 + 1}
			- n + 1}{2n}. \end{equation}
	
	For $\sigma_2(D^2u)=f$ and for scalar curvature equations, Guan--Qiu \cite{GQ} obtained an interior estimate under the condition $\sigma_3(D^2u)>-C_0$ for some constant $C_0>0$. Qiu \cite{Qiu1,Qiu} treated the three-dimensional case. More recently, Fan \cite{F} obtained estimates in dimension $n=4$ and in higher dimensions under the dynamic semiconvexity condition \eqref{eq:1.5}. Li--Wu \cite{LW} proved interior $C^2$ estimates for the quadratic Hessian equation in arbitrary dimensions. Further interior estimates for Hessian equations can be found in \cite{CHO,CCQ,HJZ,LSY}.

	Equation~\eqref{eq:1.1} is a sum Hessian equation. Such equations combine elementary symmetric functions of the Hessian and form a natural extension of the usual $k$-Hessian equation. Related equations were studied in \cite{Dong,DW,GZ,Kry,LRW,RL}. They also arise in geometric problems; see, for example, Harvey--Lawson \cite{HL80} and Espinar--G\'alvez--Mira \cite{EGM}.

	The main result is the following interior $C^2$ estimate for \eqref{eq:1.1}.
	
	\begin{theo}\label{thm:1.1}
		Let $\alpha$ satisfy \eqref{eq:1.2}, and let $f\in C^{1,1}_{\mathrm{loc}}(B_2^n\times\mathbb R\times\mathbb R^n)$ satisfy
		\[
		f(x,z,p)\ge f_0>0,\qquad \forall (x,z,p)\in B_2^n\times\mathbb R\times\mathbb R^n.
		\]
		Suppose that $u\in C^\infty(B_2^n)$ is a solution of \eqref{eq:1.1}.
		Then
		\begin{equation}\label{eq:1.6}
			\sup_{B_{1/2}^n}|D^2u|\le C.
		\end{equation}
		Here $C>0$ depends only on $n$, $\|u\|_{C^1(B_2^n)}$, $\|\alpha\|_{C^2(B_2^n)}$, $\|f\|_{C^{1,1}}$, and $\|1/f\|_{L^\infty}$.
	\end{theo}
	
	\begin{rema}
		In the paper, $C$ denotes a positive constant which may change from line to line, with its dependence specified when needed. No convexity of $u$ or integral bound for $D^2u$ is assumed.
	\end{rema}
	
	\begin{rema}
		The proof reduces \eqref{eq:1.1} to the interior Hessian estimate of Li--Wu \cite{LW}. The positivity of $f$ is needed only on the bounded range arising in the lifting argument, and the corresponding lower bound is represented by $\|1/f\|_{L^\infty}$.
	\end{rema}
	The proof uses a one-dimensional lifting and the interior estimate of Li--Wu \cite{LW}. The paper is organized as follows. Section 1 introduces the equation
	and states the main result. Section 2 records the basic identities,
	ellipticity, and the Li--Wu interior estimate. Section 3 proves
	Theorem~\ref{thm:1.1} by a one-dimensional lifting and treats the
	gradient-independent case. Section 4 gives examples and discusses
	the coefficient condition.
	
	$\text{\bf Declaration:}$
	AI was used only to assist with the counterexample in Section 4.3 and text polishing. All mathematical content was independently verified by the author. 
	
	\section{Preliminaries}
	Several elementary identities and the interior estimate for the quadratic Hessian equation are recorded first.
	
	\begin{lemm}\label{lem:2.1}
		Let $A\in\operatorname{Sym}(n)$. Then
		\begin{equation}\label{eq:2.1}
			2\sigma_2(A)=(\operatorname{tr}A)^2-|A|^2.
		\end{equation}
		Moreover, for $\gamma,d\in\mathbb R$ and $\xi\in\mathbb R^n$,
		\begin{equation}\label{eq:2.2}
			\sigma_2(A+\gamma I)
			=\sigma_2(A)+(n-1)\gamma\operatorname{tr}A
			+\frac{n(n-1)}{2}\gamma^2,
		\end{equation}
		\begin{equation}\label{eq:2.3}
			\sigma_2\begin{pmatrix} A & \xi \\ \xi^{\mathrm T} & d \end{pmatrix}
			=\sigma_2(A)+d\operatorname{tr}A-|\xi|^2.
		\end{equation}
	\end{lemm}
	
	\begin{proof}
		Since the quantities in the statement are invariant under orthogonal changes of coordinates, one may assume that $A=\operatorname{diag}(\mu_1,\cdots,\mu_n)$.
		
		\textbf{(1) Proof of \eqref{eq:2.1}.}
		Expanding the square gives
		\[
		(\operatorname{tr} A)^2
		= \biggl(\sum_{i=1}^n \mu_i\biggr)^2
		= \sum_{i=1}^n \mu_i^2 + 2\sum_{1\le i<j\le n}\mu_i\mu_j.
		\]
		Since $|A|^2=\sum_i\mu_i^2$ and $\sigma_2(A)=\sum_{i<j}\mu_i\mu_j$, it follows that
		\[
		(\operatorname{tr} A)^2 = |A|^2 + 2\sigma_2(A),
		\]
		which proves \eqref{eq:2.1}.
		
		\textbf{(2) Proof of \eqref{eq:2.2}.}
		The eigenvalues of $A+\gamma I$ are $\mu_i+\gamma$. By definition,
		\[
		\begin{aligned}
			\sigma_2(A+\gamma I)
			& = \sum_{i<j}(\mu_i+\gamma)(\mu_j+\gamma) \\
			& = \sum_{i<j}\mu_i\mu_j
			+\gamma\sum_{i<j}(\mu_i+\mu_j)
			+\gamma^2\sum_{i<j}1.
		\end{aligned}
		\]
		The first term is $\sigma_2(A)$. In the second term, each $\mu_k$ appears $n-1$ times. Hence
		\[
		\sum_{i<j}(\mu_i+\mu_j)
		=(n-1)\sum_{k=1}^n\mu_k
		=(n-1)\operatorname{tr}A.
		\]
		Also, $\sum_{i<j}1=\binom n2=n(n-1)/2$. This proves \eqref{eq:2.2}.
		
		\textbf{(3) Proof of \eqref{eq:2.3}.}
		Let
		\[
		A_0=\begin{pmatrix} A & \xi \\ \xi^{\mathrm T} & d \end{pmatrix}.
		\]
		Its trace and matrix norm satisfy
		\[
		\operatorname{tr}A_0=\operatorname{tr}A+d,\qquad
		|A_0|^2=|A|^2+2|\xi|^2+d^2.
		\]
		Substitution of these identities into \eqref{eq:2.1} yields
		\[
		\begin{aligned}
			2\sigma_2(A_0)
			& =(\operatorname{tr}A+d)^2-\bigl(|A|^2+2|\xi|^2+d^2\bigr) \\
			& =\bigl((\operatorname{tr}A)^2-|A|^2\bigr)
			+2d\operatorname{tr}A-2|\xi|^2                              \\
			& =2\sigma_2(A)+2d\operatorname{tr}A-2|\xi|^2.
		\end{aligned}
		\]
		Dividing by $2$ gives \eqref{eq:2.3}.
	\end{proof}
	\begin{lemm}\label{lem:2.2}
		Let $\Omega\subset\mathbb R^n$ be a connected open domain with $n\ge 2$. If $\alpha\in C^3(\Omega)$ and $D^2\alpha=\ell(x)I$, then $\ell$ is a constant $\beta$, and
		\[
		\alpha(x)=a_0 + b\cdot x + \frac{\beta}{2}|x|^2.
		\]
	\end{lemm}
	
	\begin{proof}
		The identity $D^2\alpha=\ell(x)I$ gives $\alpha_{ij}=0$ for $i\ne j$ and $\alpha_{ii}=\ell$. Fix $j$ and choose $i\ne j$. Since $\alpha\in C^3$, mixed third derivatives commute, and therefore
		\[
		\ell_j=\partial_j \alpha_{ii}=\partial_i \alpha_{ij}=\partial_i 0=0.
		\]
		Thus $D\ell=0$. Since $\Omega$ is connected, $\ell$ is a constant, say $\ell=\beta$.
		
		Now we define $g=\alpha-\frac{\beta}{2}|x|^2$. Then $D^2g=0$, so $Dg$ is constant. Hence $g=a_0+b\cdot x$, and the stated form of $\alpha$ follows.
	\end{proof}
	Following the notation for sum Hessian equations, we denote
	\[
	S_2(x,\lambda):=\sigma_2(\lambda)+\alpha(x)\sigma_1(\lambda).
	\]
	The admissible 
	solution set of Sum Hessian operator $S_2(x,\lambda)$ is
	\[
	\widetilde{\Gamma}_2(x)
	=\Gamma_1\cap\bigl\{\lambda\in\mathbb R^n\mid S_2(x,\lambda)>0\bigr\}.
	\]
	Thus, on the positive branch,
	\[
	\Delta u>0,\qquad
	\lambda(D^2u)\in\widetilde{\Gamma}_2(x).
	\]
	The next lemma shows the ellipticity on this admissible set.
	
	\begin{lemm}\label{lem:2.3}
		Assume $\alpha\ge0$, $f>0$, and $\Delta u>0$. Then the linearized matrix of \eqref{eq:1.1},
		\begin{equation}\label{eq:2.4}
			F^{ij}=(\Delta u+\alpha)\delta_{ij}-u_{ij}
		\end{equation}
		is positive definite.
	\end{lemm}
	
	\begin{proof}
		Fix a point and set $A=D^2u$. Define
		\[
		K=A+\frac{\alpha}{n-1}I=D^2u+\frac{\alpha}{n-1}I.
		\]
		Equations~\eqref{eq:2.2} and \eqref{eq:1.1} give
		\[
		\operatorname{tr} K=\Delta u+\frac{n\alpha}{n-1}>0,\]
		\[
		\sigma_2(K)=\sigma_2(D^2u)+\alpha\Delta u+\frac{n\alpha^2}{2(n-1)}
		=f+\frac{n\alpha^2}{2(n-1)}>0.
		\]
		Thus $K\in\Gamma_2$.
		
		Let $\lambda_i$ be an eigenvalue of $K$. By \eqref{eq:2.1},
		\[
		(\operatorname{tr} K)^2=|K|^2+2\sigma_2(K)>\lambda_i^2.
		\]
		Since $\operatorname{tr}K>0$, this implies $\operatorname{tr}K>|\lambda_i|$ and hence $\operatorname{tr}K-\lambda_i>0$.
		
		Therefore every eigenvalue of $(\operatorname{tr}K)I-K$ is positive, and this matrix is positive definite. Since
		\[
		\begin{aligned}
			(\operatorname{tr}K)I-K
			& =\left(\Delta u+\frac{n\alpha}{n-1}\right)I-\left(D^2u+\frac{\alpha}{n-1}I\right) \\
			& =(\Delta u+\alpha)I-D^2u,
		\end{aligned}
		\]
		which proves \eqref{eq:2.4}. This is exactly the derivative of $\sigma_2(A)+\alpha\operatorname{tr}A$ with respect to $A$.
	\end{proof}
	
	\begin{rema}
		This lemma gives pointwise ellipticity on the positive branch; no uniform ellipticity constant is assumed.
	\end{rema}
	The following interior Hessian estimate is due to Li--Wu \cite{LW};
	see Remarks 1.4 and 3.2 therein.
	\begin{lemm}[Li--Wu \cite{LW}]\label{lem:2.4}
		Let $n\ge2$, and let $v\in C^\infty(B_2^n)$ be a $\Gamma_2$-admissible solution of
		\[
		\sigma_2(D^2v)=\Phi(y,v,Dv)>0
		\qquad\text{in }B_2^n.
		\]
		If $\Phi\in C^{1,1}$, then
		\begin{equation}\label{eq:2.5}
			\sup_{B_{1/2}^n}|D^2v|\le C.
		\end{equation}
		Here $C>0$ depends only on $n$, $\|v\|_{C^1(B_2^n)}$, $\|\Phi\|_{C^{1,1}}$, and $\|1/\Phi\|_{L^\infty}$.
		
		If $\Phi=\Phi(y,z)$ is independent of the gradient variable, it is enough to assume $\Phi\in C^{0,1}$. In this case,
		\begin{equation}\label{eq:2.6}
			\sup_{B_{1/2}^n}|D^2v|\le C,
		\end{equation}
		where $C>0$ depends only on $n$, $\|v\|_{L^\infty(B_2^n)}$, $\|\Phi\|_{C^{0,1}}$, and $\|1/\Phi\|_{L^\infty}$.
	\end{lemm}
	
	\section{Interior \texorpdfstring{$C^2$}{C2} Estimate}
	The proof of Theorem~\ref{thm:1.1} proceeds by adding one variable and reducing \eqref{eq:1.1} to a quadratic Hessian equation.
	
	Introduce a new variable $t$. If $\alpha$ is constant, the function $u(x)+\alpha t^2/2$ has a block diagonal Hessian, and the block identity for $\sigma_2$ gives the left-hand side of \eqref{eq:1.1}. For variable $\alpha$, two extra terms appear: $t^2D^2\alpha/2$ in the $x$-block and $tD\alpha$ in the mixed block. Since $D^2\alpha=\beta I$, the unwanted Hessian term can be canceled by adding a quartic correction. We define
	\[
	U(x,t)=u(x)+\frac12\alpha(x)t^2+\eta(t),\qquad \theta=\frac{\beta t^2}{2}.
	\]
	The lower-right entry of the Hessian is $\alpha+\eta''$. By \eqref{eq:2.2}--\eqref{eq:2.3}, the coefficient of $\operatorname{tr}(D^2u)$ is
	\[
	(n-1)\theta+\alpha+\eta''(t).
	\]
	To recover the coefficient $\alpha$ in \eqref{eq:1.1}, choose
	\[
	\eta''(t)=-\frac{(n-1)\beta}{2}t^2.
	\]
	With $\eta(0)=\eta'(0)=0$, this gives
	\begin{equation}\label{eq:3.1}
		\psi(x,t):=\frac12\alpha(x)t^2-\frac{(n-1)\beta}{24}t^4,
	\end{equation}
	\begin{equation}\label{eq:3.2}
		U(x,t):=u(x)+\psi(x,t).
	\end{equation}
	Since $\alpha_{ij}=\beta\delta_{ij}$, the derivatives are
	\begin{equation}\label{eq:3.3}
		\psi_i=\frac12\alpha_i t^2,
	\end{equation}
	\begin{equation}\label{eq:3.4}
		\psi_t=\alpha t-\frac{(n-1)\beta}{6}t^3,
	\end{equation}
	\begin{equation}\label{eq:3.5}
		\psi_{ij}=\frac{\beta t^2}{2}\delta_{ij},
	\end{equation}
	\begin{equation}\label{eq:3.6}
		\psi_{it}=t\alpha_i,
	\end{equation}
	\begin{equation}\label{eq:3.7}
		\psi_{tt}=\alpha-\frac{(n-1)\beta}{2}t^2.
	\end{equation}
	Here and below, $i,j$ denote the $x$-variables. Thus
	\begin{equation}\label{eq:3.8}
		D_{x,t}^2U=
		\begin{pmatrix}
			D_x^2u+\frac{\beta t^2}{2}I & tD\alpha                        \\
			t(D\alpha)^T                & \alpha-\frac{(n-1)\beta t^2}{2}
		\end{pmatrix}.
	\end{equation}
	\begin{lemm}\label{lem:3.1}
		For the function defined in \eqref{eq:3.1}--\eqref{eq:3.2},
		\begin{equation}\label{eq:3.9}
			\sigma_2\big(D_{x,t}^2U\big)=\sigma_2\big(D_x^2u\big)+\alpha(x)\Delta_x u+E(x,t),
		\end{equation}
		where
		\begin{equation}\label{eq:3.10}
			E(x,t)=t^2\left(\frac{n\beta\alpha(x)}{2}-|D\alpha(x)|^2\right)-\frac{n(n-1)\beta^2}{8}t^4.
		\end{equation}
		In particular, $E$ depends only on $x$ and $t$.
	\end{lemm}
	\begin{proof}
		Let $A=D_x^2u$ and $\theta=\beta t^2/2$. By \eqref{eq:2.2}--\eqref{eq:2.3},
		\[
		\begin{aligned}
			\sigma_2(D^2U)
			& =\sigma_2(A+\theta I)+\big(\alpha-(n-1)\theta\big)\operatorname{tr}(A+\theta I)-t^2|D\alpha|^2             \\
			& =\sigma_2(A)+(n-1)\theta\operatorname{tr}A+\frac{n(n-1)}{2}\theta^2                                           \\
			& \quad +\alpha\operatorname{tr}A-(n-1)\theta\operatorname{tr}A+n\alpha\theta-n(n-1)\theta^2-t^2|D\alpha|^2.
		\end{aligned}
		\]
		The two terms containing $(n-1)\theta\operatorname{tr}A$ cancel. Hence
		\[
		\sigma_2(D^2U)=\sigma_2(A)+\alpha\operatorname{tr}A+n\alpha\theta-\frac{n(n-1)}{2}\theta^2-t^2|D\alpha|^2.
		\]
		Substituting $\theta=\beta t^2/2$ gives
		\[
		n\alpha\theta=\frac{n\beta\alpha}{2}t^2,
		\]
		\[
		\frac{n(n-1)}{2}\theta^2=\frac{n(n-1)\beta^2}{8}t^4,
		\]
		and the result follows.
	\end{proof}
	Let $X=(x,t)\in\mathbb R^{n+1}$ and $P=(p,\tau)\in\mathbb R^n\times\mathbb R$. Define
	\begin{equation}\label{eq:3.11}
		G(X,z,P):=f\bigl(x,z-\psi(x,t),p-D_x\psi(x,t)\bigr)+E(x,t).
	\end{equation}
	Since
	\[
	U-\psi=u,\qquad D_xU-D_x\psi=D_xu,
	\]
	it follows along the lifted solution that
	\begin{equation}\label{eq:3.12}
		\sigma_2(D_{x,t}^2U)=G(X,U,D_{x,t}U).
	\end{equation}
	
	For $|t|\le1$, the explicit formulas \eqref{eq:3.1} and \eqref{eq:3.10} give
	\begin{equation}\label{eq:3.13}
		|\psi|+|D_{x,t}\psi|\le C,
		\qquad
		|E(x,t)|\le Ct^2.
	\end{equation}
	Here $C>0$ depends only on $n$ and $\|\alpha\|_{C^2(B_2^n)}$.
	Consequently, on $B_2^n\times(-1,1)$,
	\begin{equation}\label{eq:3.14}
		|U|+|D_{x,t}U|\le C.
	\end{equation}
	Here $C>0$ depends only on $n$, $\|u\|_{C^1(B_2^n)}$, and $\|\alpha\|_{C^2(B_2^n)}$.  From now on, unless otherwise stated, \(C>0\) denotes a positive constant depending only on \(n\), \(\|u\|_{C^1(B_2^n)}\), \(\|\alpha\|_{C^2(B_2^n)}\), \(\|f\|_{C^{1,1}}\), and \(\|1/f\|_{L^\infty}\); it may change from line to line.
	
	By the estimate for $E$, there exists $h\in(0,1]$, depending only on
	$n$, $\|\alpha\|_{C^2(B_2^n)}$, and $\|1/f\|_{L^\infty}$, such that
	\begin{equation}\label{eq:3.15}
		|E(x,t)|
		\le
		\frac{1}{2\|1/f\|_{L^\infty}}
		\qquad
		\text{in }\Omega_h:=B_2^n\times(-h,h).
	\end{equation}
	Hence, throughout this cylinder and for all \((z,P)\) in the bounded range considered below,
	\begin{equation}\label{eq:3.16}
		G(X,z,P)
		\ge
		\frac{1}{2\|1/f\|_{L^\infty}}>0.
	\end{equation}
	
	\begin{lemm}\label{lem:3.2}
		Under the assumptions of Theorem~\ref{thm:1.1}, $U$ is $\Gamma_2$-admissible on $\Omega_h$.
	\end{lemm}
	\begin{proof}
		At $t=0$, \eqref{eq:3.8} gives
		\[
		D^2U(x,0)=
		\begin{pmatrix}
			D_x^2u(x) & 0         \\
			0         & \alpha(x)
		\end{pmatrix},
		\]
		and hence
		\[
		\operatorname{tr}D^2U(x,0)=\Delta_xu(x)+\alpha(x)>0,
		\]
		while
		\[
		\sigma_2(D^2U(x,0))=f(x,u,Du)>0.
		\]
		By \eqref{eq:3.12} and \eqref{eq:3.16},
		\[
		\sigma_2(D^2U)>0\qquad\text{in }\Omega_h.
		\]
		Using \eqref{eq:2.1},
		\begin{equation}\label{eq:3.17}
			\bigl(\operatorname{tr}D^2U\bigr)^2
			=|D^2U|^2+2\sigma_2(D^2U)>0.
		\end{equation}
		Thus $\operatorname{tr}D^2U$ never vanishes in $\Omega_h$. For each fixed $x$, it is positive at $t=0$ and continuous in $t$, so it remains positive for $|t|<h$. Therefore
		\[
		\sigma_1(D^2U)>0,
		\qquad
		\sigma_2(D^2U)>0,
		\]
		and hence $D^2U\in\Gamma_2$ on $\Omega_h$.
	\end{proof}
	
	Next, we define
\begin{equation}\label{eq:3.18}
	T(X,z,P)
	:=
	\bigl(x,z-\psi(x,t),p-D_x\psi(x,t)\bigr),
\end{equation}
where
\[
X=(x,t)\in\mathbb R^{n+1},
\qquad
P=(p,\tau)\in\mathbb R^n\times\mathbb R.
\]
Notice that only the first $n$ components $p$ of $P$ enter the
definition of $T$, since the original right-hand side $f$ depends on
the $x$-gradient $D_xu$. By definition,
\[
(f\circ T)(X,z,P)
=
f\bigl(x,z-\psi(x,t),p-D_x\psi(x,t)\bigr).
\]
Hence the right-hand side in \eqref{eq:3.12} takes the form
\[
G(X,z,P)
=
(f\circ T)(X,z,P)+E(x,t).
\]

We first establish the $C^{1,1}$ regularity of $G$ on a bounded
neighborhood of the range of the lifted solution. Since
\[
D^2\alpha=\beta I,
\]
the function $\alpha$ is quadratic. Consequently, all derivatives of
$\psi$ which arise when differentiating $T$ are controlled by
$\|\alpha\|_{C^2(B_2^n)}$. In particular,
\[
\psi_{ijk}=0,\qquad
\psi_{ijt}=\beta t\delta_{ij},\qquad
\psi_{itt}=\alpha_i,\qquad
\psi_{ttt}=-(n-1)\beta t.
\]
Together with the first and second derivative formulas for $\psi$,
these identities imply
\begin{equation}\label{eq:3.19}
	\sup\bigl(|DT|+|D^2T|\bigr)\le C.
\end{equation}
Since \(f\in C^{1,1}_{\mathrm{loc}}\), \(f\circ T\) is also \(C^{1,1}\) on this set. Indeed, the first
derivatives satisfy
\[
D(f\circ T)
=
(Df\circ T)\,DT,
\]
while, at points of twice differentiability, the second derivatives
are combinations of terms of the form
\[
(D^2f\circ T)(DT,DT)
\qquad\text{and}\qquad
(Df\circ T)D^2T.
\]
Thus the $C^{1,1}$ bound for $f$, together with
\eqref{eq:3.19}, yields
\begin{equation}\label{eq:3.20}
	\|f\circ T\|_{C^{1,1}}\le C.
\end{equation}

We next estimate the error term. Recall from \eqref{eq:3.10} that
\[
E(x,t)
=
t^2\left(
\frac{n\beta\alpha(x)}{2}
-|D\alpha(x)|^2
\right)
-\frac{n(n-1)\beta^2}{8}t^4.
\]
Since $\alpha$ is quadratic, all of its derivatives of order greater
than two vanish. Direct differentiation of the above expression
therefore gives
\begin{equation}\label{eq:3.21}
	\|E\|_{C^{1,1}(\Omega_h)}\le C.
\end{equation}
Since
\[
G=f\circ T+E,
\]
equations \eqref{eq:3.20} and \eqref{eq:3.21} imply
\begin{equation}\label{eq:3.22}
	\|G\|_{C^{1,1}}\le C
\end{equation}
on a bounded neighborhood of the range of
\[
(X,U,D_{x,t}U).
\]

The function $G$ is needed only near this range. In order to apply
Lemma~\ref{lem:2.4}, however, it is convenient to extend the
right-hand side in the variables $(z,P)$ so that the extension remains
positive and has a uniform $C^{1,1}$ bound.

Choose a smooth cutoff function $\chi=\chi(z,P)$ satisfying
\[
0\le\chi\le1,
\]
such that $\chi=1$ on the range of
\[
(U,D_{x,t}U)
\]
described by \eqref{eq:3.14}, and such that the support of $\chi$ is
contained in a slightly larger bounded domain on which
\eqref{eq:3.16} holds. Define
\begin{equation}\label{eq:3.23}
	\widetilde G(X,z,P)
	:=
	\chi(z,P)G(X,z,P)
	+
	\bigl(1-\chi(z,P)\bigr)
	\frac{1}{2\|1/f\|_{L^\infty}}.
\end{equation}

Since $\chi=1$ along the lifted solution,
\[
\widetilde G(X,U,D_{x,t}U)
=
G(X,U,D_{x,t}U).
\]
Thus replacing $G$ by $\widetilde G$ does not alter the equation
satisfied by $U$.

Moreover, on the support of $\chi$, equation \eqref{eq:3.16} gives
\[
G\ge
\frac{1}{2\|1/f\|_{L^\infty}},
\]
whereas outside this support the second term in
\eqref{eq:3.23} is exactly the same positive constant. Since
$0\le\chi\le1$, it follows that
\begin{equation}\label{eq:3.24}
	\widetilde G
	\ge
	\frac{1}{2\|1/f\|_{L^\infty}}
	>0.
\end{equation}
The smoothness of $\chi$ and \eqref{eq:3.22} further imply
\begin{equation}\label{eq:3.25}
	\|\widetilde G\|_{C^{1,1}}\le C.
\end{equation}

We now localize the lifted equation near an arbitrary point
$x_0\in B_{1/2}^n$. We define
\[
X_0:=(x_0,0)
\]
and
\begin{equation}\label{eq:3.26}
	\rho:=\frac h4.
\end{equation}
Since $h\le1$, we have $\rho\le1/4$. If
\[
X=(x,t)\in B_{2\rho}^{n+1}(X_0),
\]
then
\[
|t|<2\rho=\frac h2<h
\]
and
\[
|x|
\le |x_0|+|x-x_0|
<\frac12+2\rho
\le1.
\]
Therefore
\[
B_{2\rho}^{n+1}(X_0)\subset\Omega_h.
\]
Since $h$ was chosen only in terms of the quantities entering the
estimate,
\begin{equation}\label{eq:3.27}
	\rho^{-1}\le C.
\end{equation}
In particular, the localization radius $\rho$ is bounded from below
uniformly in $x_0$.

To rescale $B_{2\rho}^{n+1}(X_0)$ onto the fixed ball
$B_2^{n+1}$, define
\begin{equation}\label{eq:3.28}
	w(Y)
	:=
	\rho^{-2}U(X_0+\rho Y),
	\qquad
	Y\in B_2^{n+1}.
\end{equation}
A direct differentiation gives
\begin{equation}\label{eq:3.29}
	D_Yw(Y)
	=
	\rho^{-1}D_XU(X_0+\rho Y),
	\qquad
	D_Y^2w(Y)
	=
	D_X^2U(X_0+\rho Y).
\end{equation}
The second identity shows that the Hessian is unchanged by this
normalization. Hence the $\Gamma_2$-admissibility of $U$ established
in Lemma~\ref{lem:3.2} is inherited by $w$.

Define the rescaled right-hand side by
\begin{equation}\label{eq:3.30}
	\Phi_{X_0}(Y,z,P)
	:=
	\widetilde G
	\bigl(X_0+\rho Y,\rho^2z,\rho P\bigr).
\end{equation}
Since
\[
U(X_0+\rho Y)=\rho^2w(Y)
\]
and, by \eqref{eq:3.29},
\[
D_XU(X_0+\rho Y)
=
\rho D_Yw(Y),
\]
the equation for $U$ becomes
\begin{equation}\label{eq:3.31}
	\sigma_2(D_Y^2w)
	=
	\Phi_{X_0}(Y,w,D_Yw)
	\qquad\text{in }B_2^{n+1}.
\end{equation}

It remains to verify the bounds required by
Lemma~\ref{lem:2.4}. From \eqref{eq:3.14},
\eqref{eq:3.27}, and \eqref{eq:3.29},
\[
|w(Y)|
\le
\rho^{-2}\sup|U|,
\qquad
|D_Yw(Y)|
\le
\rho^{-1}\sup|D_XU|.
\]
Therefore
\begin{equation}\label{eq:3.32}
	\|w\|_{C^1(B_2^{n+1})}\le C.
\end{equation}

For the right-hand side, \eqref{eq:3.30} is obtained from
$\widetilde G$ by the affine change of variables
\[
(Y,z,P)
\longmapsto
(X_0+\rho Y,\rho^2z,\rho P).
\]
Since $0<\rho\le1$, differentiation with respect to $Y$, $z$, and
$P$ introduces only factors bounded by $1$. Hence
\eqref{eq:3.24} and \eqref{eq:3.25} yield
\begin{equation}\label{eq:3.33}
	\|\Phi_{X_0}\|_{C^{1,1}}
	+
	\left\|
	\frac1{\Phi_{X_0}}
	\right\|_{L^\infty}
	\le C.
\end{equation}
The constants in \eqref{eq:3.32} and \eqref{eq:3.33} are uniform
with respect to $x_0$.

All hypotheses of the Li--Wu interior Hessian estimate are now
satisfied. Indeed, $w$ is $\Gamma_2$-admissible by
Lemma~\ref{lem:3.2} and \eqref{eq:3.29}, equation
\eqref{eq:3.31} has a positive right-hand side by
\eqref{eq:3.24}, and the required $C^{1,1}$ bounds follow from
\eqref{eq:3.32}--\eqref{eq:3.33}. Applying
Lemma~\ref{lem:2.4} in dimension $n+1$ gives
\[
|D_Y^2w(0)|\le C.
\]
By \eqref{eq:3.29},
\[
D_Y^2w(0)
=
D_X^2U(x_0,0),
\]
and therefore
\[
|D_X^2U(x_0,0)|\le C.
\]

Finally, setting $t=0$ in \eqref{eq:3.8} gives
\[
D_X^2U(x_0,0)
=
\begin{pmatrix}
	D_x^2u(x_0) & 0\\
	0 & \alpha(x_0)
\end{pmatrix}.
\]
Hence
\[
|D_x^2u(x_0)|
\le
|D_X^2U(x_0,0)|
\le C.
\]
Since $x_0\in B_{1/2}^n$ was arbitrary,
\[
\sup_{B_{1/2}^n}|D^2u|\le C.
\]
This proves \eqref{eq:1.6} and completes the proof of
Theorem~\ref{thm:1.1}.
\qed
	
	\begin{coro}[The gradient-independent case]\label{cor:3.3}
		Assume \eqref{eq:1.2}, and let
		\[
		f=f(x,z)\in C^{0,1}_{\mathrm{loc}}(B_2^n\times\mathbb R),
		\qquad f>0.
		\]
		If $u\in C^\infty(B_2^n)$ is a solution of \eqref{eq:1.1} satisfying
		\[
		\Delta u>0\qquad\text{in }B_2^n,
		\]
		then
		\begin{equation}\label{eq:3.34}
			\sup_{B_{1/2}^n}|D^2u|\le C.
		\end{equation}
		Here $C>0$ depends only on $n$, $\|u\|_{L^\infty(B_2^n)}$, $\|\alpha\|_{C^2(B_2^n)}$, $\|f\|_{C^{0,1}}$, and $\|1/f\|_{L^\infty}$. 
	\end{coro}
	
	\begin{proof}
		In this case \eqref{eq:3.11} becomes
		\[
		G(X,z)=f\bigl(x,z-\psi(x,t)\bigr)+E(x,t),
		\]
		which is independent of the gradient variable. The construction above uses only $\|u\|_{L^\infty(B_2^n)}$ to control $U$, and the cutoff can be chosen to depend only on $z$. Hence the extended right-hand side remains independent of the gradient variable and satisfies
		\[
		\|\widetilde G\|_{C^{0,1}}
		+\left\|\frac1{\widetilde G}\right\|_{L^\infty}\le C.
		\]
		Here $C>0$ depends only on the quantities listed in Corollary~\ref{cor:3.3}.
		After the same localization and scaling, the gradient-independent part of Lemma~\ref{lem:2.4} gives \eqref{eq:3.34}.
	\end{proof}

	\section{Examples and Remarks on the Coefficient Condition}
	
	\subsection{Affine coefficients}
	If $\alpha(x)=a_0+b\cdot x\ge0$, then $\beta=0$. In this case the lifting becomes
	\begin{equation}\label{eq:4.1}
		U(x,t)=u(x)+\frac12\alpha(x)t^2,
		\qquad
		\sigma_2(D^2U)=f(x,u,Du)-t^2|b|^2.
	\end{equation}
	Since $f$ has a positive lower bound on the bounded range used in the lifting argument, the lifted right-hand side remains positive after restricting to a sufficiently thin strip in the $t$-direction. No smallness assumption on $b$ is needed; the width of the strip and the final estimate depend on $|b|$ through $\|\alpha\|_{C^2}$.
	
	Examples on $B_2^n$ satisfying Theorem~\ref{thm:1.1} include
	\[
	\alpha(x)=3+x_1,\qquad \alpha(x)=1+|x|^2,
	\qquad \alpha(x)=5-|x|^2.
	\]
	These examples cover the affine, convex quadratic, and concave quadratic cases. The key assumption is $D^2\alpha=\beta I$, not the convexity or concavity of $\alpha$.
	
	\subsection{Why a direct shift is not enough}
	For an arbitrary smooth $\alpha$, the identity
	\begin{equation}\label{eq:4.2}
		\sigma_2\!\left(D^2u+\frac{\alpha(x)}{n-1}I\right)
		=f(x,u,Du)+\frac{n\alpha(x)^2}{2(n-1)}
	\end{equation}
	still holds. One may try to absorb the added scalar matrix into the Hessian of a new function. If
	$D^2\varphi=\alpha(x)I/(n-1)$, then for any $j$ and any $i\ne j$,
	\[
	\frac{\partial_j\alpha}{n-1}=\partial_j\varphi_{ii}
	=\partial_i\varphi_{ij}=0.
	\]
	Hence, on a connected domain, such a function $\varphi$ exists only when $\alpha$ is constant. The lifting used above avoids this problem by introducing an extra variable and allowing mixed derivatives.
	
	For a general coefficient $\alpha$, consider the simpler lifting $U=u+\alpha(x)t^2/2$. Let $A=D^2u$ and $B=D^2\alpha$. A direct calculation gives
	\begin{align}\label{eq:4.3}
		\sigma_2(D^2U)
		={} & f+\frac{t^2}{2}\bigl((\tr B)\tr A-\tr(BA)\bigr)\notag \\
		& +\frac{\alpha t^2}{2}\tr B-t^2|D\alpha|^2
		+\frac{t^4}{4}\sigma_2(B).
	\end{align}
	If $B=\beta I$, the term containing $D^2u$ reduces to $(n-1)\beta t^2\Delta u/2$, and the quartic correction cancels it. For a general $B$, the term $\operatorname{tr}(BA)$ remains. Thus the condition $D^2\alpha=\beta I$ is needed for this lifting. This calculation does not give a necessary condition for other methods.
	
	\subsection{An \texorpdfstring{$L^\infty$}{L-infinity} bound on the coefficient is not enough}
	The following example shows that an $L^\infty$ bound on $\alpha$ alone is not sufficient for an interior $C^2$ estimate in a general coefficient class.
	
	\begin{lemm}\label{lem:4.1}
		For every $n\ge2$, there exist smooth functions $\alpha_\varepsilon>0$ and
		$u_\varepsilon$ on $B_2^n$ such that
		\[
		\sigma_2(D^2u_\varepsilon)+\alpha_\varepsilon\Delta u_\varepsilon=1,
		\qquad \Delta u_\varepsilon>0,
		\]
		and $|\alpha_\varepsilon|_{L^\infty}+
		|u_\varepsilon|_{C^1}$ is uniformly bounded, whereas
		$(u_\varepsilon)_{11}(0)\to\infty$.
	\end{lemm}
	
	\begin{proof}
		For $0<\varepsilon<1$, we define
		\[
		u_\varepsilon(x):=x_1^2+\sqrt{x_1^2+\varepsilon^2}
		-\frac{1}{2(n-1)}\sum_{i=2}^n x_i^2,
		\qquad \kappa:=\frac{n}{2(n-1)},
		\]
		\[
		L_\varepsilon:=2+
		\frac{\varepsilon^2}{(x_1^2+\varepsilon^2)^{3/2}},
		\qquad s_\varepsilon:=\Delta u_\varepsilon=L_\varepsilon-1>0.
		\]
		Then
		\[
		D^2u_\varepsilon=\operatorname{diag}
		\left(L_\varepsilon,-\frac1{n-1},\cdots,-\frac1{n-1}\right),
		\quad
		\sigma_2(D^2u_\varepsilon)+s_\varepsilon=-\kappa.
		\]
		Define
		\[
		\alpha_\varepsilon(x):=1+\frac{1+\kappa}{s_\varepsilon(x)}.
		\]
		Then $1<\alpha_\varepsilon\le2+\kappa$ and
		$\sigma_2(D^2u_\varepsilon)+\alpha_\varepsilon s_\varepsilon=1$.
		Moreover, $u_\varepsilon$ and $Du_\varepsilon$ are uniformly bounded in $B_2^n$, while
		\[
		(u_\varepsilon)_{11}(0)=2+\varepsilon^{-1}\longrightarrow\infty.
		\]
	\end{proof}
	
	For this family, the derivatives of $\alpha_\varepsilon$ are not uniformly bounded. In fact,
	\[
	\partial_{11}\alpha_\varepsilon(0)
	=\frac{3(1+\kappa)}{\varepsilon(1+\varepsilon)^2}
	\longrightarrow\infty.
	\]
	Thus the example does not rule out estimates that also depend on a $C^{1,1}$ bound for $\alpha$. The present argument does not address whether \eqref{eq:1.2} can be replaced by a more general structural assumption.

\end{document}